\documentclass{amsart}

\usepackage{amsmath, amssymb, amsthm}
\usepackage[mathcal]{eucal}
\usepackage{enumerate}
\usepackage{verbatim, hyperref}
\usepackage{xcolor}
\usepackage{appendix}

\newtheorem{theorem}{Theorem}

\newtheorem*{theorem*}{Theorem}
\newtheorem{lemma}{Lemma}

\newtheorem{proposition}[theorem]{Proposition}

\def\ZR{\ensuremath{\mathcal R}}

\def\ZP{\ensuremath{\mathcal P}}

\def\Z{\ensuremath{\mathbb Z}}
\def\C{\ensuremath{\mathbb C}}

\def\R{\ensuremath{\mathbb R}}
\def\T{\ensuremath{\mathbb T}}
\def\D{\ensuremath{\mathbb D}}

\renewcommand{\tilde}{\widetilde}
\renewcommand{\hat}{\widehat}
\renewcommand{\epsilon}{\varepsilon}
\renewcommand{\phi}{\varphi}
\renewcommand{\bar}{\overline}

\newcommand{\Id}{\operatorname{Id}}

\newcommand{\supp}{\operatorname{supp}}

\def\Zchi{\ensuremath{\mathbf 1}}

\def\ti{\tilde}

\numberwithin{equation}{section}

\author{Gevorg Mnatsakanyan}
\address{Institute of Mathematics of the National Academy of Sciences of Armenia, Marshal Baghramyan Ave. 24/5, Yerevan 0019, Armenia}
\curraddr{}
\email{mnatsakanyan\_g@yahoo.com}

\title[$\ell^1$ mapping properties of NLFT]{$\ell^1$ mapping properties, smoothness and decay for $SU(2)$-valued nonlinear Fourier transform}

\date{March 2026}

\subjclass{81P68,34L25,42C99, 47B35}

\begin{document}

\begin{abstract}
We prove an analog of Baxter's theorem for $SU(2)$-valued nonlinear Fourier transform (NLFT). That is, we prove that under certain natural conditions on the NLFT data, the potential is in $\ell^1$ if and only if the linear Fourier coefficients of the NLFT data are in $\ell^1$. Furthermore, we prove some smoothness-decay estimates for the NLFT motivated by similar estimates for the linear Fourier transform.
\end{abstract}

\maketitle

\section{Introduction}
For a finitely supported sequence $F=(F_n)_{n\in \Z}$ of complex numbers let us define the product
\begin{equation}\label{inifinite product nlft}
    G(z) = \prod_{k=-\infty}^{\infty} \frac{1}{\sqrt{1+|F_k|^2}} \begin{pmatrix}
    1 & F_k z^k \\
    -\bar F_k z^{-k} & 1
\end{pmatrix} \, ,    
\end{equation}
where $z$ is on the unit circle $\T$ and the infinite product is ordered such that the terms with lower indices appear to the left.

For a function $f$ defined on a subset $X$ of Riemann sphere we will denote by $f^*$ its holomorphic reflection across the circle $\T$. That is, $f^*$ is defined on $X^*= \{\bar z^{-1} : z\in X\}$ by $f^*(z) = \overline{f(\bar z^{-1})}$. For $z\in \T$, $f^*(z)$ is simply $\overline{f(z)}$.

Each matrix on the right-side of \eqref{inifinite product nlft} is in the group $SU(2)$, hence, so is $G(z)$ and we can write it as
$$
G(z) = \begin{pmatrix}
    a(z) & b(z) \\
    -b^* (z) & a(z) 
\end{pmatrix}\, ,
$$
with
\begin{equation}\label{determinant identity}
   |a(z)|^2 + |b(z)|^2 = 1 \, . 
\end{equation}
We will shorthand the $SU(2)$ matrices by their first row, so, for example, we write $(a,b)$ for $G$. We call $(a,b)$ the $SU(2)$-valued nonlinear Fourier transform (NLFT) of $F$. We emphasize that it is the $SU(2)$-valued NLFT in contrast to its more well-known $SU(1,1)$-valued kin \cite{TaoThiele2012}. For the $SU(1,1)$ case one only needs to change in \eqref{inifinite product nlft} the minus to a plus in the lower-left entry of the matrices and the plus to a minus in the scalar factor.

The definition of the NLFT can be extended to $\ell^2$ sequences by taking appropriate limits for the NLFT of the truncated sequences.

The $SU(2)$-valued NLFT was first systematically studies in \cite{tsai} which develops $\ell^1$ and $\ell^2$ theories of the NLFT and treats cases of rational or soliton data. More recently, the $SU(2)$-NLFT sparked our interest due to its connection to quantum computing. It turns out that an algorithm in quantum computing called quantum signal processing (QSP) \cite{low2017} coincides up to a change of variables with the $SU(2)$-NLFT for purely imaginary sequences $F$ \cite{QSP_NLFA}. One of the main problems of QSP is determining the coefficients $F$ given the function $b$. It was recently solved in \cite{QSP5author} for the functions $b$ satisfying the Szeg\"o condition which is the largest possible class where the $SU(2)$-NLFT is known to be well-defined. The methods of \cite{QSP5author} rely on Riemann-Hilbert factorization techniques and spectral theory. We develop this program further and obtain $SU(2)$ variant of Baxter's theorem \cite{baxter} for orthogonal polynomials on the unit circle (OPUC) and smoothness-decay estimates for the NLFT. These demonstrate that $\ell^2$ theory, $\ell^1$ theory and Sobolev estimates can be viewed from a unified Riemann-Hilbert framework.

To state our first main theorem we need a little notation. For $f\in L^2(\T)$, let $\hat{f}\in \ell^2 (\Z)$ denote the sequence of its linear Fourier coefficients. We call a function $w\,: \, \Z\to [1,\infty)$ a strong Beurling weight if the following conditions hold:
\begin{itemize}
    \item[i.] $w(-n) = w(n)$ for all $n$,
    \item[ii.] $w(n+m)\leq w(n)w(m)$ for all $n,m$.
    \item[iii.] $\lim_{n\to \infty} \frac{\log w(n)}{n} = 0 $ .
\end{itemize}
The standard examples of strong Beurling weights are $w(n) = (1+|n|)^{\alpha}$ for $\alpha \geq 0$.

For a strong Beurling weight $w$ we define the weighted $\ell^1$ space
$$
\ell_w^1 := \{a=(a_{n})_{n\in \Z}\,:\, \|a\|_w:= \sum_{n\in \Z} |a_n|w(n) < \infty \} \, .
$$
With the standard convolution operation, $\ell^1_w$ is a Banach algebra which by Gelfand transform is isomorphically and isometrically mapped to
$$
A_w := \{f:\T\to \C \,:\, \hat{f} \in \ell^1_w \} \, ,
$$
where the multiplication operation is now the pointwise multiplication. As for the algebra norms, let us register here that
$$
\|f\|_{A_w} := \| \hat{f} \|_{\ell_w^1} \, .
$$
When $w\equiv 1$, $A_1$ is just the usual Wiener algebra.

Baxter's theorem \cite{baxter, golinski} is usually stated in the language of OPUC or the Schur function. However, in the equivalent language of the $SU(1,1)$-NLFT, it states that $F\in \ell^1_w$ if and only if $b/a^* \in A_w$ and $\|b/a^*\|_{L^\infty (\T)} < 1$. Our $SU(2)$-variant is as follows.
\begin{theorem}\label{main thm}
Assume $(a,b)$ is the NLFT image of the sequence $F=(F_j)_{j\in \Z} \in \ell^2 (\Z)$ and $w$ is a strong Beurling weight. If $F\in \ell^1_w$, then $a^*,b \in A_w$. Conversely, if $a^*$ is outer on $\D$ and $b/a^* \in A_w$, then $F\in \ell_w^1$.

\end{theorem}

The forward direction, that is $F\in \ell^1_w$ implies $b, a^*\in A_w$ is easy and follows from the multilinear expansion of the NLFT discussed below in Section \ref{preliminaries}. The opposite direction is more difficult. The analogous direction for OPUC relies on the orthogonality structure to relate the NLFT of the truncated sequence $(F_j)_{j=0}^N$ with the NLFT of the full sequence $F$ through a truncated Toeplitz operator. Then, one applies the functional calculus of $A_w$ to obtain an $LU$ factorization of the Toeplitz operator to be able to invert the truncated operator.
In $SU(2)$, the orthogonality structure is available only in the limited range of $\|b\|_{L^\infty(\T)}<1/\sqrt{2}$ as demonstrated in \cite{onesidedOPUC}. Instead, we apply the Riemann-Hilbert factorization techniques. We largely follow the proof of Baxter's theorem from \cite[Chapter 5]{simon}. However, our method allows us to get a simpler and more direct $LU$ factorization without using the functional calculus. Another minor advantage is that we allow the sequence $F$ to be supported on all of $\Z$ instead of the half-line $\Z_+$ which is a natural restriction when working with the OPUC. Let us also mention that for the $SU(1,1)$-NLFT the function $a$ is always outer, whereas in $SU(2)$, $a$ may have an inner factor. It is this fact that is responsible for the existence of solitons and for making $SU(2)$ theory in many respects more difficult. The study of the outer $a$ for $SU(2)$ was motivated by QSP theory. Nevertheless, we think that Theorem \ref{main thm} should be true at least with finite Blaschke factors in $a$. We hope to address this case in future work. 

It would be interesting to see what kind of quantification is possible for Theorem \ref{main thm} at least in the unweighted case $w\equiv 1$. It is known \cite{QSP_NLFA, Linlin} that the forward direction is easily quantifiable. Indeed, we will see in the proof of Theorem \ref{main thm} that
\begin{equation}\label{b wiener algebra bound}
    \|b\|_{A_1} \leq \sinh (\|F\|_{\ell^1}) \, .
\end{equation}
In the opposite direction, the best result, to the best of our knowledge, is obtained in \cite{Linlin} in the context of QSP. It states that there is a universal constant $r \approx 0.902$ such that if $\|b\|_{A_1}<r$ then there exists a sequence $F$ such that $b$ is the second entry of its NLFT and
\begin{equation}\label{lin estimate}
    \|F\|_{\ell^1} \lesssim \|b\|_{A_1}\, . 
\end{equation}
In fact, this result is obtained for QSP, so when $F_n$'s are purely imaginary. However, we think that the proof carries over to the general case.

In an effort to quantify Theorem \ref{main thm}, we establish the following estimate:
\begin{proposition}\label{quantitative root 2}
Let $(a,b)$ be the $SU(2)$-NLFT of a sequence $F\in \ell^2(\Z)$ such that $a^*$ is outer in $\D$. Let $w$ be a strong Beurling weight and assume $\|b\|_{A_w}<1/\sqrt{2}-\varepsilon$, then
\begin{equation}
    \|F\|_{\ell^1_w} \lesssim \frac{1}{\varepsilon} \|b/a\|_{A_w} \, .
\end{equation}
\end{proposition}
We note that the range $\|b\|_{A_1} < 1/\sqrt{2}$ is more restrictive than the one obtained in \cite{Linlin}. However, our approach is entirely independent of multilinear expansions. Furthermore, the threshold $1/\sqrt{2}$ appears naturally in the study of $SU(2)$-NLFT within the context of Riemann-Hilbert problem and orthogonal polynomials as discussed in \cite{QSP_NLFA, onesidedOPUC}.

Next, we consider smoothness-decay type results motivated by the linear theory. For an $m$ times continuously differentiable function $f:\T\to \C$, it is well-known and easy to obtain from integration by parts that
\begin{equation}\label{linear smoothness}
    |\hat{f}(n)|\leq |n|^{-m} \|f^{(m)}\|_{L^1(\T)} \, .  
\end{equation}
For $s>0$ we define the fractional Sobolev space $H^s(\T)$ as the set of $f\in L^2(\T)$ for which
$$
\|f\|_{H^s}:= \sum_{n\in \Z} (1+|n|^2)^s |\hat{f}(n)|^2 < \infty \, .
$$
Then, our second theorem is as follows.
\begin{theorem}\label{thm decay}
    Let $(a,b)$ be the NLFT image of the sequence $F\in \ell^2(\Z)$ and let $a^*$ be outer in $\D$. Additionally, let $\frac{b}{a^*} \in H^s(\T)$ for some $s\geq 1$ and $(\frac{b}{a^*})' \in L^{\infty}(\T)$. Then,
    \begin{equation}\label{fractional decay}
        |F_n| \lesssim_s \frac{a^*(0)}{|n|^s} (1+\| b/a^* \|_{H^{s}(\T)} ) \max \left( 1, \left\| \left(\frac{b}{a^*} \right)' \right\|_{L^\infty(\T)}^{\lceil s\rceil} \right) \, ,
    \end{equation}
    where $\lceil s\rceil$ is the smallest integer not less than $s$.
    
    If $s=1$, we have the better estimate
    \begin{equation}\label{firstorder decay}
        |F_n| \leq \frac{2a^*(0)}{|n|} \|(b/a^*)'\|_{L^2(\T)} \,.
    \end{equation}
\end{theorem}

Similar results for OPUC were obtain in \cite{golinski} by classical but intricate techniques and relying on orthogonality of the polynomials. An alternative approach is to apply Baxter's theorem with weights $w(n)=(1+|n|)^{\alpha}$. That is, if $b/a^*$ is $m$ times continuously differentiable and satisfies the hypothesis of Theorem \ref{main thm} (or in the $SU(1,1)$ case that of Baxter's theorem), then applying the linear estimate \eqref{linear smoothness} and Theorem \ref{main thm} will yield
$$
\sum_n |F_n| |n|^{\beta} < \infty
$$
for any $\beta<m-1$. For OPUC, this is discussed in \cite[Corollary 5.2.3]{simon}. Similar decay and tail estimates are also obtained in \cite{Linlin} in a restricted range as for \eqref{lin estimate}.

Our approach again relies on the Riemann-Hilbert factorization method. Mimicking the linear theory, we substitute partial integration by the fractional Leibniz rule of Kato-Ponce \cite{Kato1988CommutatorEA}. Compared to \cite[Theorm 3]{golinski}, our method yields a cleaner, shorter proof, with explicit dependence on the NLFT data and a better estimate for Sobolev regularity of order $\geq 1$. On the other hand, \cite[Theorem 3]{golinski} is more subtle when $(b/a^*)^{(m)}$ has modulus of continuity of slower than H\"older type decay.

In Section \ref{preliminaries}, we give a brief overview of NLFT theory. In Section \ref{sec proof of thm1} we prove theorem 1 and proposition \ref{quantitative root 2} and in section \ref{sec decay-smoothness} we prove Theorem \ref{thm decay}.

\subsection*{Acknowledgments}

I thank Michel Alexis and Kristina Oganesyan for numerous helpful discussions. In particular, the initial version of Theorem \ref{thm decay} was conceived in conversations with Michel Alexis.

This work was supported by the Higher Education and Science Committee of Republic of Armenia (Research Project No 24RL-1A028).

\section{Preliminaries on $SU(2)$-nonlinear Fourier transform}\label{preliminaries}
In this section we give a brief overview of the facts from the theory of $SU(2)$-NLFT. For a more in-depth discussion we refer to \cite{QSP_NLFA,tsai}.

The definition \eqref{inifinite product nlft} of the NLFT can be given recursively. We can define the sequence of $SU(2)$-valued functions $G_k = (a_k,b_k) : \T \to SU(2)$ by
\begin{equation}\label{recursion for a}
    a_{k} = \frac{1}{\sqrt{1+|F_k|^2}} \left( a_{k-1}(z) - \bar F_k z^{-k} b_{k-1} \right) \, ,
\end{equation}
and
\begin{equation}\label{recursion for b}
    b_k = \frac{1}{\sqrt{1+|F_k|^2} }\left( b_{k-1} + F_kz^k a_{k-1} \right) \, ,
\end{equation}
with the initial condition 
$$\lim_{k\to -\infty} a_k = 1 \text{ and } \lim_{k\to -\infty} b_k = 0 \, .$$
Then, for finitely supported $F$, $G_k$ is a stationary sequence so we put $G = \lim_{k\to +\infty} G_k$. For sequences in $\ell^1$ or $\ell^2$ this limit can be shown to exist in the appropriate spaces.

We can expand the product in \eqref{inifinite product nlft} by the distributive law,
$$
(a,b) = \prod_{k=-\infty}^{\infty} \frac{1}{\sqrt{1+|F_k|^2}} \left( (1,0)
+ (0, F_k z^k) \right)
$$
and obtain an expansion of the functions $a$ and $b$ as a sum of multilinear operators of $F$,
\begin{equation}\label{multilinear expansion}
    = \prod_{k=-\infty}^{\infty} \frac{1}{\sqrt{1+|F_k|^2}} \begin{pmatrix}
    \sum_{n=0}^\infty T_{2n} (F,F,\dots F)(z), & \sum_{n=0}^\infty T_{2n+1} (F,F,\dots F)(z) \\
\end{pmatrix} \, ,    
\end{equation}
where the $n$-linear operator is defined for sequences $F^1,\dots, F^n$ by
\begin{equation}\label{multilinear term}
    T_n (F^1,\dots, F^n)(z) := \sum_{j_1<j_2<\dots < j_n} \left(\prod_{\substack{1\le k\le n \\k \text{ is odd}}}F^k_{j_k}z^{j_k} \right) \left ( \prod_{\substack{1\le k\le n \\ k \text{ is even}}}-\overline{F^k_{j_k}} z^{-j_k}\right )  \, .
\end{equation}
The above $n$-linear operator can be estimated in $A_1$ norm as,
\begin{equation}
    \|T_n(F^1,\dots,F^n)\|_{A_1} \leq \frac{1}{n!}\prod_{k=1}^n \|F^k\|_{\ell^1} \, .
\end{equation}
Summing these estimates for even and odd $n$, when $F^k\equiv F$ for all $k$, we get a Wiener algebra bound for $a$ and $b$ as in \eqref{b wiener algebra bound} for $w\equiv 1$.

It is straightforward that for finitely supported sequences $F$, $a$ and $b$ are Laurentz polynomials in $z$, and so can be defined on all of $\C\setminus \{0\}$. A closer look reveals that for an interval $(n,m]\subset \Z$,
\begin{equation}\label{supp of interval}
    \supp \hat{a^*} \subset [0, m-n) \text{ and } \supp \hat{b} \subset (n,m] \, .
\end{equation}
So $a^*$ can always be extended as a bounded holomorphic function in $\D$. Unlike for $SU(1,1)$-NLFT, where $a^*$ is always an outer function in $\D$, for $SU(2)$-NLFT $a^*$ may have zeros or a singular inner part. The type of NLFT data when $a^*$ is an inner function and $b$ is $0$ are called solitons.

For $F\in \ell^2 (\Z)$, one can see that
\begin{equation}\label{a^*0}
    \prod_{k\in \Z} (1+|F_k|^2)^{-\frac{1}{2}} = a^*(0) \, . 
\end{equation}
When $a$ is outer we can write $\log a^*(0)=\log |a^* (0)|$ through the mean of $\log|a|$ on $\T$. Recalling \eqref{determinant identity},
\begin{equation}\label{plancherel}
    \sum_{k\in \Z} \log (1+|F_k|^2) = -\frac{1}{2\pi} \int_{-\pi}^\pi \log (1-|b(e^{i\theta})|^2) d\theta \, .
\end{equation}
We call \eqref{plancherel}, the Szeg\"o-Plancherel identity.

Next, let us discuss the Riemann-Hilbert factorization problem.
\subsection{Riemann-Hilbert factorization}\label{subsec rh}
Henceforth, let $(a,b)$ be the $SU(2)$-NLFT of some sequence $F\in \ell^2 (\Z)$, let $a^*$ be outer in $\D^*$ and $\|b\|_{L^\infty (\T)} < 1$.
Note that the hypotheses of both Theorem \ref{main thm} and \ref{thm decay}, specifically $\|b/a^*\|_{A_1} < \infty$ and $b/a^*\in H^s$ with $s\geq 1$, imply that $b/a^*$ is continuous and so bounded from above. Thus, by \eqref{determinant identity}, $\|b\|_{L^\infty(\T)}<1$ and $|a|$ is bounded from below away from zero.

For $I\subset \mathbb{Z}$, let $P_{I}:L^2(\T)\to L^2(\T)$ be the projection operator such that for any $f\in L^2(\T)$ we have $\widehat{P_I f} = \hat{f} \cdot \Zchi _I$. Let $P_+ := P_{[0,\infty)}$ and $P_-:= \Id-P_+=P_{(-\infty,0)}$.

Recall, $(a_n,b_n)$ is the NLFT of the truncated sequence $(F_j)_{j=-\infty}^{n}$ and let $(a_{n,+},\, b_{n,+})$ be the NLFT of the sequence $(F_j)_{j=n+1}^\infty$. From \eqref{inifinite product nlft}, we have
$$
(a,\,b) = (a_n,\,b_n)\cdot (a_{n,+},\, b_{n,+}) \, ,
$$
where each of the pairs stands for an $SU(2)$ matrix. From here, we rewrite
$$
(a_{n,+},\, b_{n,+}) = (a_n,\, b_n)^{-1} \cdot (a,\, b) = (a_n^* ,\, -b_n ) \cdot (a,\, b) \, .
$$
Entrywise this reads
$$
a_{n,+} = a_n^* a + b_n b^* \quad \text{ and }\quad
b_{n,+} = a_n^* b - b_n a^* \, .
$$
We divide the first equation by $a$ and the second equation by $a^*$. Then, projecting the first one on $P_+ L^2(\T)$ and the second one on $P_{\leq n}L^2(\T)$, we can write
\begin{equation}\label{rh projec 1}
    P_+ \left( \frac{a_{n,+}}{a} \right) = P_+ (a_n^*) + P_+ \left( \frac{b^*}{a} b_n \right) \, ,    
\end{equation}
and
\begin{equation}\label{rh projec 2}
    P_{\leq n} \left( \frac{b_{n,+}}{a^*} \right) = - P_{\leq n} (b_n) + P_{\leq n } \left( \frac{b}{a^*} a_n^* \right) \, .
\end{equation}
As $|a|$ is bounded away from $0$, $1/a$ is also in $H^\infty (\D^*)$, so $a_{n,+}/a \in H^\infty (\D^*)$. Therefore, 
$$P_+\left(\frac{a_{n,+}}{a}\right) = \frac{a_{n,+}(\infty)}{a(\infty)} = \frac{1}{a_n(\infty)}\, ,$$
where the last equality follows from \eqref{a^*0}. From \eqref{supp of interval}, $\supp \hat{b_{n,+}} \subset [n+1,\infty)$ and $\supp \hat{b_n}\subset (-\infty,n]$. So, $\supp \hat{b_{n,+}/a^*} \subset [n+1,\infty)$ and $P_{\leq n}(b_{n,+}/a^*) = 0$. So, equations \eqref{rh projec 1} and \eqref{rh projec 2} become
\begin{equation}\label{rh 1}
    a_n^* = 1/a_n^*(0) - P_{+}(\frac{b^*}{a}b_n) \, ,    
\end{equation}
\begin{equation}\label{rh 2}
    b_n=P_{\leq n} (\frac{b}{a^*}a_n^*) \, .
\end{equation}

Let $\Id_2$ be the $2\times 2$ identity matrix and let 
$$
M = \begin{pmatrix}
        0 & P_+ \frac{b^*}{a} P_{\leq n} \\
        - P_{\leq n} \frac{b}{a^*} P_+ & 0
    \end{pmatrix} \, ,
$$
In the matrix form, we can write \eqref{rh 1} and \eqref{rh 2} as
\begin{equation}\label{rh matrix L2 form}
    (\Id_2+M) \begin{pmatrix}
        a_n^* \\ b_n
    \end{pmatrix} = \begin{pmatrix}
        1/a_n^* (0) \\ 0
    \end{pmatrix} \, .
\end{equation}
In \cite{QSP5author}, the above equation was considered for $\begin{pmatrix}
    a_n^* \\ b_n
\end{pmatrix}$ in the Hilbert space $L^2(\T)\times L^2(\T)$. Then, $M$ is a bounded anti-symmetric operator on the Hilbert space, hence, has pure imaginary spectrum. Therefore, $\Id_2+M$ is invertible with
\begin{equation}\label{inverse op est}
    \| (\Id_2+M)^{-1}\|_{2\to 2} \leq 1 \, .
\end{equation}
So, given $(a,b)$ we can find $(a_n,b_n)$ and prove the existence of a sequence $F$ that is the inverse NLFT image of $(a,b)$. We remark that \cite{QSP5author} deals with the general case when the restriction $\|b\|_{L^\infty (\T)}<1$ is dropped. There one must make use of the spectral theory for unbounded operators.

For the purposes of Theorem \ref{main thm}, we will need to consider the equation \eqref{rh matrix L2 form} in the Banach space $A_w\times A_w$. Whereas for Theorem \ref{thm decay}, we will again be in the realm of the Hilbert space estimate \eqref{inverse op est}.


\section{Proof of Theorem \ref{main thm}}\label{sec proof of thm1}

\subsection{The easy part}
Recall the multilinear expansion \eqref{multilinear expansion}. We can rewrite  the $2n+1$-linear term \eqref{multilinear term} as
$$
T_{2n+1} (F,\dots, F) = (-1)^n \sum_{k\in \Z} z^k \sum\limits_{\substack{j_1<\cdots < j_{2n+1} \\ j_{2n+1}-j_{2n}+\cdots + j_1 = k }} F_{2n+1} \prod_{l=1}^n \bar F_{2l-1} \bar F_{2l} \, .
$$
By condition i and ii of strong Beurling weights
$$
\prod_{l=1}^{2n+1} w(j_l)= \prod_{l=1}^{2n+1} w((-1)^{l-1} j_l) \geq w( \sum_{l=1}^{2n+1} (-1)^l j_{l-1} ) = w(k) \, .
$$
Hence, we can estimate
$$
\| T_{2n+1} (F,\dots, F)\|_{A_w} \leq \sum_{k\in\Z} w(k) \left|\sum\limits_{\substack{j_1<\cdots < j_{2n+1} \\ j_{2n+1}-j_{2n}+\cdots + j_1 = k }} F_{2n+1} \prod_{l=1}^n \bar F_{2l-1} \bar F_{2l} \right|
$$
$$
\leq \sum_{j_1<\cdots < j_{2n+1}} \prod_{l=1}^{2n+1} |F_{l}|w(l) \leq \frac{\|F\|_{\ell^1_w}^{2n+1} }{(2n+1)!} \, .
$$
Putting together the above estimate with a triangle inequality, we get
$$
\|b\|_{A_w} = \left\| \sum_{n=0}^\infty T_{2n+1} (F,\dots, F) \right\|_{A_w}
\leq \sum_{n=0}^{\infty} \frac{\|F\|_{\ell^1_w}^{2n+1}}{(2n+1)!} = \sinh (\|F\|_{\ell^1_w}) \, .
$$
Similar multlinear estimate yields $a\in A_w$. This proves the easy part of Theorem \ref{main thm}.

\subsection{Some Notation and Preliminaries}
Define the Banach space of column vectors
$$
Y_w := \left\{ \begin{pmatrix}
    y_1 \\ y_2
\end{pmatrix} \, : \, y_1,y_2\in A_w  \right\}
$$
with norm
$$
\|y\|_{Y_w} = \left\| \begin{pmatrix}
    y_1 \\ y_2
\end{pmatrix}\right\|_{Y_w} = \|y_1\|_{A_w}+\|y_2\|_{A_w} \, .
$$
We will be considering operators on the Banach space $Y_w$. Recall the projection operators $P_I$ from subsection \ref{subsec rh} and define projections on $Y_w$,
$$
\ZP_+ = \begin{pmatrix}
    P_+ & 0 \\
    0 & \Id 
\end{pmatrix} \, ,\quad  \ZP_- = \begin{pmatrix}
    P_- & 0 \\
    0 & 0
\end{pmatrix} \, , \quad  \ZP_n = \begin{pmatrix}
    P_{+} &0 \\
    0 & P_{\leq n}
\end{pmatrix} \, , \quad \ZR_n = \ZP_+-\ZP_n \, .
$$
Also let
$$
C := \begin{pmatrix}
    1 & \frac{b^*}{a} \\
    -\frac{b}{a^*} & 1
\end{pmatrix} \, .
$$
$C$ is a bounded operator on $Y_w$. With this notation, the equation \eqref{rh matrix L2 form} takes the form
\begin{equation}\label{rh matrix form}
    \ZP_n C \ZP_n \begin{pmatrix}
    a_n^* \\ b_n
\end{pmatrix} = \begin{pmatrix}
    1/a_n^*(0) \\ 0
\end{pmatrix} \, .    
\end{equation}
Let us define the renormalization of $a_n$ and $b_n$,
$$\ti A_n = a_n^*(0) a_n\, ,\quad \ti B_n = a_n^*(0)b_n\, ,\quad A_n = a_n/a_n^*(0)\, ,\, \quad B_n = b_n/a_n^*(0) \, .$$
So \eqref{rh matrix form} becomes
\begin{equation}\label{main eq: line-halfline}
    \ZP_n C \ZP_n \begin{pmatrix}
    \ti A_n^* \\ \ti B_n
\end{pmatrix} = \begin{pmatrix}
    1 \\ 0
\end{pmatrix} \, .
\end{equation}

The following is Wiener's lemma for the commutative Banach algebra $A_w$ and is standard. We give the proof for completeness.
\begin{lemma}[Wiener's lemma for $A_w$]\label{Wiener's lemma for A_w}
If $f\in A_w$ and $0\notin f(\T)$, then $1/f \in A_w$.
\end{lemma}

\begin{proof}
Let us show that for $f\in A_w$, the spectrum of $f$ in $A_w$ is given by $\sigma_{A_w}(f) = f(\T)$. Then, the lemma follows from \cite[Theorem 11.5]{rudin}. As $A_w$ is a commutative Banach algebra,
$$
\sigma(f) = \{h(f) \, :\, h \text{ bounded complex homomorphism of }A_w \} \, .
$$
That point evaluations on $\T$ are bounded complex homomorphisms is clear, so $\sigma(f) \subset f(\T)$. To see the opposite inclusion, let $h$ be an arbitrary bounded complex homomorphism $h$. We have for $g(\theta)=\sum_{j\in \Z} c_j e^{ij\theta} \in A_w$,
$$
h(g) = \sum_{j\in \Z} c_j h(e^{i\theta})^j \, .
$$
Thus, we need to understand the possible values of $h(e^{i\theta})$. For $n\geq 0$
$$
|h(e^{i\theta})|^n = |h(e^{in\theta})| \leq \| e^{in\theta}\|_w = w(n) \, .
$$
Hence, $|h(e^{i\theta})|\leq e^{\frac{1}{n}\log w(n)}$ for all $n$. Taking the limit as $n\to \infty$, we conclude $|h(e^{i\theta})|\leq 1$ by property iii of strong Beurling weights. Similarly estimating $|h(e^{-in\theta})|$, we obtain $|h(e^{i\theta})|\geq 1$ and we are done.
\end{proof}

The following simple lemma relies on Lemma \ref{Wiener's lemma for A_w}.
\begin{lemma}\label{equivalent assumptions ab in algebra}
Let $w$ be a strong Beurling weight, $a$ be outer on $\D^*$ and $|a|^2+|b|^2=1$. Then, the following are equivalent
\begin{itemize}
    \item[(a)] $a,b \in A_w$ and $\| b\|_{L^\infty (\T)} < 1$,
    \item[(b)] $b/a^* \in A_w$.
\end{itemize}
\end{lemma}
\begin{proof}
We will use Wiener's lemma \ref{Wiener's lemma for A_w} for $A_w$. Assuming $(a)$, we have $1/a^* \in A_w$, then also $b\cdot \frac{1}{a^*} \in A_w$. 
Next, assuming $(b)$, we have $\frac{b^*}{a} \in A_w$, so also $\frac{|b|^2}{|a|^2} \in A_w$, adding $1$ the function we have $\frac{1}{|a|^2} \in A_w$. As $1/|a|^2 \geq 1$, by Wiener's lemma $|a|^2\in A_w$. Furthermore, $1/|a|^2$ is continuous and so bounded. We deduce $|a|^2$ is also continuous and away from $0$. So we can define $\log |a|^2 \in A_w$ by functional calculus which will coincide with the pointwise defined logarithm \cite[Theorem 10.30]{rudin}.   $\log |a|^2 = 2\log |a|$ is also in $A_w$. Then, as the Hilbert transforms maps $A_w$ to itself, so $H(\log |a|) = \arg a \in A_w$. Hence, $\log a = \log |a|+i\arg a \in A_w$ and $a = e^{\log a} \in A_w$. Then, $b= \frac{b}{a^*} \cdot a^* \in A_w$.
\end{proof}

\subsection{The difficult part}
Our main task will be to show that $a_n,b_n$ are in $A_w$ for large enough $n$ and that $a_n\to a$ and $b_n\to b$ in $A_w$ as $n\to+\infty$. This is the content of Lemma \ref{main step lemma} below. For that we will need to invert the operator $\ZP_nC\ZP_n$ which is the content of Lemma \ref{inverting rhoCrho}, which is accomplished by manipulating the $LU$ factorization for $C$. So let us start from the $LU$ factorization.

Given a Banach space $K$, let $\Id$ be the identity operator on $K$ and $Q:K\to K$ be a projection. We call a bounded linear operator $T:K\to K$ upper-triangular with respect to $Q$ if $QT(\Id-Q) = 0$ and lower-triangular if $(\Id-Q)TQ=0$.

Define
$$
L := \begin{pmatrix}
    \frac{1}{a} & \frac{b^*}{a} \\
    0 & 1
\end{pmatrix} \, \quad\quad U :=\begin{pmatrix}
    \frac{1}{a^*} & 0 \\
    -\frac{b}{a^*} & 1
\end{pmatrix} \, ,
$$
and
$$
\ti L := \begin{pmatrix}
    1 & \frac{b^*}{a} \\
    0 & \frac{1}{a}
\end{pmatrix} \, \quad\quad
\ti U :=\begin{pmatrix}
    1 & 0 \\
    -\frac{b}{a^*} & \frac{1}{a^*} 
\end{pmatrix} \, ,
$$

\begin{lemma}\label{LU factorization for C}
We have
\begin{equation}\label{wienerhopd for C}
    C=LU=\ti U\ti L \, ,  
\end{equation}
and $L, L^{-1}, \ti L, \ti L^{-1}$ are lower-triangular, $U, U^{-1}, \ti U, \ti U^{-1}$ are upper-triangular with respect to projection $\ZP_+$ on $Y_w$. Furthermore, we have
\begin{equation}\label{0 proj with L and U}
    \ZR_n \ti L \ZP_n = \ZR_n \ti L^{-1} \ZP_n = \ZP_n\ti U \ZR_n = \ZP_n\ti U^{-1} \ZR_n = 0 \, .
\end{equation}
\end{lemma}
\begin{proof}
Equations \eqref{wienerhopd for C} are straightforward and only use $|a|^2+|b|^2=1$. Let us prove the second part.
Consider
$$
\ZP_+ L \ZP_- = \ZP_+ \begin{pmatrix}
    \frac{1}{a} P_- & 0 \\
    0 & 0 
\end{pmatrix} = \begin{pmatrix}
    P_+ \frac{1}{a}P_- & 0 \\
    0 & 0
\end{pmatrix} \, .
$$
As $\frac{1}{a}$ is supported on Fourier side on $(-\infty,0]$,we have $P_+ \frac{1}{a} P_-= 0$. Similar computations holds for the other seven operators. Equalities \eqref{0 proj with L and U} are also straightforward. Let us check one of them.
$$
\ZR_n \ti L \ZP_n = \begin{pmatrix}
    0 & 0\\
    0 & P_{>n} 
\end{pmatrix} \begin{pmatrix}
    1 & \frac{b^*}{a} \\
    0 & \frac{1}{a}
\end{pmatrix} \begin{pmatrix}
    P_+ & 0\\
    0 & P_{\leq n}
\end{pmatrix}
$$
$$
=\begin{pmatrix}
    0 & 0 \\
    0 & P_{>n} \frac{1}{a} P_{\leq n}
\end{pmatrix} = 0 \, . 
$$
\end{proof}

\begin{lemma}\label{inverting rhoCrho}
Let $n$ be such that $\|\frac{P_{>n}(b)}{a^*}\|_{A_w}<\frac{1}{2}$, then the operator $\ZP_nC\ZP_n$ is invertible on $\ZP_nY_w$ with
$$
\|(\ZP_nC\ZP_n)^{-1} \|_{\ZP_nY_w\to \ZP_n Y_w} \leq (\|a\|_{A_w}+2\|1/a\|_{A_w})(\|a\|_{A_w}+\|b\|_{A_w})^2 \, .
$$
\end{lemma}
Note, that $\|\frac{P_{>n}(b)}{a^*}\|_{A_w}\to 0$ as $n\to +\infty$, so the hypothesis of the lemma holds for large enough $n$.

\begin{proof}
Take arbitrary
$$y = \begin{pmatrix}
    y_1 \\ y_w
\end{pmatrix}\in \ZP_n Y_w \, ,$$
that is $y_1\in P_+ A_w$ and $y_2\in P_{\leq n}A_w$. We need to show
\begin{itemize}
    \item[(a)] there exists $x = \begin{pmatrix}
        x_1\\x_2
    \end{pmatrix}\in \ZP_nY_w$ such that $\ZP_nC\ZP_n x = y$,
    \item[(b)] $\|x\|_{Y_w}\lesssim \|y\|_{Y_w}$.
\end{itemize}

Let us assume that such $x$ exists and deduce a system of equations that will be easier to solve. Let us put $w = \begin{pmatrix}
    0 \\ w_2
\end{pmatrix} =(\ZP_+-\ZP_n)Cx$ and $u = \begin{pmatrix}
    u_1 \\ 0
\end{pmatrix} = \ZP_-Cx$. We have $w_2\in P_{>n}A_w$ and $u_1\in P_-A_w$ and
\begin{equation}\label{eq: Cx e w u}
    Cx = y + w + u \, .
\end{equation}
Then,
\begin{equation}\label{Ux L-1}
    Ux = L^{-1}y + L^{-1}w + L^{-1}u \, .
\end{equation}
By lemma \ref{LU factorization for C}, $\ZP_-Ux = \ZP_-U\ZP_+x = 0$ and $\ZP_+ L^{-1} u = \ZP_+L^{-1} \ZP_-u = 0$, so projecting the above equation with $\ZP_-$, we get
$$
\ZP_- L^{-1} w + L^{-1}u = -\ZP_- L^{-1}y \, .
$$
Writing the above equation coordinatewise, we get
$$-P_- (b^* w_2) + au_1 = -P_-(ay_1-b^*y_2) \, .
$$
Since $w_2\in P_{>n}A_w$, $P_- (b^*w_2) = P_- \big( P_{<-n}(b^*) w_2 \big)$, we rewrite
\begin{equation}\label{coordinatewise 1}
    -P_- ( P_{<-n}(b^*) w_2) + au_1 = -P_-(ay_1-b^*y_2) \, .
\end{equation}
On the other hand, from \eqref{eq: Cx e w u}, we can write
$$
\ti L x = \ti U^{-1} y + \ti U^{-1} w + \ti U^{-1} u \, .
$$
Now projecting with $\ZR_n$, we get
$$
\ti U^{-1}w + \ZR_n \ti U^{-1} u = - \ZR_n \ti U^{-1} y \, ,
$$
as $\ZR_n \ti L x = \ZR_n \ti L \ZP_n x = 0$ and $\ZR_n \ti U^{-1} w = \ti U^{-1} w $.
Again, coordinatewise we have
$$
a^*w_2 + P_{>n} (bu_1) = - P_{>n} (by_1+a^*y_2) \, .
$$
As $u_1\in P_-A_w$, $P_{>n} (bu_1) = P_{>n}\big( P_{>n}(b) u_1 \big)$, so
\begin{equation}\label{coordinatewise 2}
    a^*w_2 + P_{>n} ( P_{>n}(b) u_1) = - P_{>n} (by_1+a^*y_2) \, .
\end{equation}
We want to write equations \eqref{coordinatewise 1} and \eqref{coordinatewise 2} in the matrix form. Let
$$
T_1 = \begin{pmatrix}
    a & 0 \\
    0 & a^*
\end{pmatrix} \, \text{ and } T_2 = \begin{pmatrix}
        0 & - P_- P_{<-n} (b^*) \\
        P_{>n} P_{>n}(b)
    \end{pmatrix} \, ,
$$
then $T_1$ and $T_2$ are both operators on $P_-A_w\oplus P_{>n}A_w$, and
\begin{equation}\label{matrix form T1+T_2}
    (T_1 + T_2) \begin{pmatrix}
        u_1 \\ w_2
    \end{pmatrix} = \begin{pmatrix}
        - P_- (ay_1-b^*y_2)    
        \\
        -P_{>n} (by_1+a^*y_2)
    \end{pmatrix} \, .
\end{equation}
Furthermore, we have by hypothesis of the lemma,
$$\|T_1^{-1}T_2\|_{P_-A_w\oplus P_{>n}A_w\to P_-A_w\oplus P_{>n}A_w} < \frac{1}{2} \, .$$
So $T_1+T_2$ is invertible on $P_-A_w\oplus P_{>n}A_w$ with
$$
\|(T_1+T_2)^{-1}\|_{P_-A_w\oplus P_{>n}A_w\to P_-A_w\oplus P_{>n}A_w} \leq \frac{\|T_1^{-1}\|}{1-\|T_1^{-1}T_2\|} \leq 2\|1/a\|_{A_w} \, .
$$
So we can estimate
$$
\|u_1\|_{A_w} + \|w_2\|_{A_w} \leq 2\|1/a\|_{A_w} \left\|\begin{pmatrix}
    - P_- (ay_1-b^*y_2)    
        \\
        -P_{>n} (by_1+a^*y_2)
\end{pmatrix}\right\|_{Y_w}
$$
$$\leq 2\|1/a\|_{A_w} \big( \|a\|_{A_w} + \|b\|_{A_w}\big) \|y\|_{Y_w} \, .
$$
The last inequality immediately implies part $(b)$. From \eqref{eq: Cx e w u}, we estimate
$$
\|x\|_{Y_w} = \| C^{-1}y\|_{Y_w} + \| C^{-1}w\|_{Y_w}+\| C^{-1}u\|_{Y_w}
$$
\begin{equation}\label{norm estimate x y}
    \leq ( \||a|^2\|_{A_w}+\|ab\|_{A_w} ) \big( 1 + 2\|1/a\|_{A_w}(\|a\|_{A_w}+\|b\|_{A_w})\big) \|y\|_{Y_w} \, .    
\end{equation}

To prove part (a), it suffices to show that if we start with $y$, and define $u_1$ and $w_2$ by \eqref{matrix form T1+T_2}, then
$$
x := C^{-1} (y+w+u)
$$
is in $\ZP_nY_w$. Let us write the equation for the first coordinates above
$$
x_1 = |a|^2 y_1-a^*b^* y_2 -a^*b^* w_2+|a|^2u_1
$$
$$
= a^* \Big( a y_1 - b^* y_2 - b^*w_2+a u_1 \Big) \, .
$$
By \eqref{coordinatewise 1}, the expression in the parenthesis above is in $P_+A_w$, hence $x_1\in P_+A_w$. For the second coordinates, we have
$$
x_2 = ab y_1 + |a|^2y_2 +|a|^2w_2 + abu_1
$$
$$
= a\Big( by_1+a^*y_2 + a^*w_2+bu_1 \Big) \, .
$$
By \eqref{coordinatewise 2}, the expression in the last parenthesis above is in $P_{\leq n}A_w$, thus so is $x_2$. As $x_1\in P_+A_w$, $x_2\in P_{\leq n}A_{w}$, we have $x\in \ZP_nY_w$ as wanted.
\end{proof}

From Lemma \ref{inverting rhoCrho} it is easy to deduce our main lemma.
\begin{lemma}\label{main step lemma}
There exists $n_0=n_0(a,b)$ such that for $n>n_0$ $a_n,b_n\in A_w$ and $a_n\to a$ and $b_n\to b$ in $A_w$ as $n\to +\infty$.
\end{lemma}
\begin{proof}
We choose $n_0$ so that $\| \frac{P_{>n}(b)}{a^*}\|_{A_w}<\frac{1}{2}$. Then, Lemma \ref{inverting rhoCrho} implies that for $n>n_0$, we have the solution $\begin{pmatrix}
    \ti A_n^* \\ \ti B_n
\end{pmatrix} \in Y_w$ satisfying \eqref{main eq: line-halfline}. To show the convergence, let us write
$$
\left\| \begin{pmatrix}
    \ti A_n^* \\ \ti B_n
\end{pmatrix} - \begin{pmatrix}
    \ti A^* \\ \ti B
\end{pmatrix} \right\|_{Y_w} = \left\| \big( (\ZP_n C\ZP_n)^{-1} - (\ZP_+ C \ZP_+)^{-1} \big) \begin{pmatrix}
    1 \\0
\end{pmatrix} \right\|_{Y_w}
$$
$$
= \left\| (\ZP_nC\ZP_n)^{-1} ( \ZP_+ C \ZP_+ - \ZP_n C \ZP_n ) (\ZP_+C\ZP_+)^{-1} \begin{pmatrix}
    1 \\ 0
\end{pmatrix} \right\|_{Y_w}
$$
$$
\leq \|(\ZP_nC\ZP_n)^{-1}\|_{Y_w\to Y_w} \left\|( \ZP_+ C \ZP_+ - \ZP_n C \ZP_n ) (\ZP_+C\ZP_+)^{-1} \begin{pmatrix}
    1 \\ 0 
\end{pmatrix} \right\|_{Y_w} \, .
$$
The latter converges to $0$ as $n\to \infty$, as for any $v\in Y_w$, $\lim_{n\to \infty} (\ZP_+-\ZP_n)v= 0$ in $Y_w$ and $\|\ZP_n\|_{Y_w\to Y_w}\leq 1$.

Finally, note that
$$a_n^*(0) = \sqrt{\ti A_n^*(0)} = \left( \frac{1}{2\pi} \int_{-\pi}^{\pi}\ti A_n^*(e^{i\theta}) d\theta \right)^{\frac{1}{2}} \to \left( \frac{1}{2\pi} \int_{-\pi}^{\pi}\ti A^*(e^{i\theta}) d\theta \right)^{\frac{1}{2}} = a^*(0) \, ,$$
as $n\to \infty$, hence the conclusion of the lemma.
\end{proof}

To get from Lemma \ref{main step lemma} to the conclusion of Theorem \ref{main thm}, we use what is called Baxter's trick. The recursion \eqref{recursion for b} written for the renormalizations $A_n,B_n$ take the form
\begin{equation}\label{monopolynomial recursion with B}
    B_{n+1} = B_n + F_{n+1} z^{n+1}A_n \, .
\end{equation}
Let us rewrite the last display as
$$
B_n = -B_{n+1} + F_{n+1}z^{n+1}A_n \, .
$$
Let $A = a / a^*(0)$ and $B=b/a^*(0)$, and multiply both sides of the above display by $A^{-1}$. We have
$$
A^{-1}B_n = -A^{-1}B_{n+1} + F_{n+1}z^{n+1}A_nA^{-1} \, .
$$
As $A^{-1}$ is in $H^\infty (\D^*)$ and $\hat{B_n}\subset (-\infty,n]$, the left side above is of the form $\sum_{j=-\infty}^n c_jz^j$ and the right side is of the form $\sum_{j=-\infty}^{n+1} d_jz^j$. Hence, $d_{n+1}=0$. We can see that the coefficient of $z^{n+1}$ in $F_{n+1}z^{n+1}A_n A^{-1}$ is equal to $F_{n+1} A_n(\infty)A^{-1}(\infty) = F_{n+1}$. Thus,
$$
\|A^{-1} B_n\|_{A_w} \leq \| A^{-1} B_{n+1}\|_{A_w} + |F_{n+1}| \|A^{-1} A_n\|_{A_w} - 2 w(n+1) |F_{n+1}| \, .
$$
Adding these inequalities for $n=0,\dots, N$, we have a telescoping sum,
$$
\|A^{-1}B_{0}\|_{A_w} \leq \| A^{-1} B_{N+1}\|_{A_w} + \sum_{n=1}^{N+1} |F_n| (\|A^{-1}A_{n-1}\|_{A_w}-2w(n)) \, .
$$
By Lemma \ref{main step lemma}, $A_N\to A$ and $B_N\to B$ as $N\to \infty$, hence, also $A^{-1}A_N \to 1$ and $A^{-1}B_{N+1}\to A^{-1}B$ in $A_w$. Thus, $\sum_{n\geq 1} |F_n| w(n) <\infty$.

By \cite[Theorem 2]{QSP_NLFA}, $(a^*(z^{-1}),b(z^{-1}))$ is the NLFT image of $(F_{-n})_{n\in \Z}$. Thus, to show that also $\sum_{n\leq 0} |F_n|$, we apply the same argument to $(a(z), b(z^{-1}))$.

\subsection{Proof of Proposition \ref{quantitative root 2}}
Manipulating the equations in the proof of Lemma \ref{inverting rhoCrho} we get
\begin{lemma}
If $\|b\|_{A_w}<1$, then
\begin{equation}
    \left\| \frac{A_n^*}{A^*} \right\|_{A_w} \leq \frac{1}{1- \|b\|_{A_w}^2} \, .
\end{equation}
\end{lemma}
\begin{proof}
Equation \eqref{Ux L-1} written for $y=\begin{pmatrix}
    1\\0
\end{pmatrix}$ becomes
$$
U \begin{pmatrix}
    \ti A_n^* \\ \ti B_n
\end{pmatrix} = L^{-1}\begin{pmatrix}
    1 \\ 0
\end{pmatrix} + L^{-1}u + L^{-1}w \, .
$$
We know that as $U$ is upper triangular so $\ZP_-U\ZP_+ = 0$. Similarly, $L^{-1}$ is lower triangular and $u\in \ZP_-Y_w$, so $\ZP_+L^{-1}u=0$. Projecting the above equation with $\ZP_+$, we get
$$
U \begin{pmatrix}
    \ti A_n^* \\ \ti B_n
\end{pmatrix} =\ZP_+ L^{-1}\begin{pmatrix}
    1 \\ 0
\end{pmatrix} + \ZP_+ L^{-1}\begin{pmatrix}
    0 \\ w_2
\end{pmatrix} \, .
$$
Entrywise, we have
$$
\frac{\ti A_n^*}{a^*} = a^* (0) - P_+ (b^*v_2) \, ,
$$
and
$$
w_2 = -P_{>n} (\frac{b}{a^*} \ti A_n^*) \, .
$$
Substituting the $w_2$ in the first equation,
$$
\frac{\ti A_n^*}{\ti A^*} = 1 + P_+ (b^* P_{>n} (b \frac{\ti A_n^*}{\ti A^*})) \, .
$$
Therefore,
\begin{equation}
    \left\| \frac{\ti A_n}{\ti A} \right\|_{A_w} \leq \frac{1}{1- \|b\|_{A_w}^2} \, .
\end{equation}
Recalling, $0< a^*(0)\leq a_n^*(0)$, we obtain the conclusion of the lemma.
\end{proof}

If $\|b\|_{A_w}<1/\sqrt{2}-\varepsilon$, we have $\|A_n/A\|_{A_w} < 2-4\varepsilon$ for all $n$. So Baxter's trick applies for all $n$, and we get
$$
\sum_{n\in \Z} |F_n|w(n) \lesssim \frac{1}{\varepsilon} \|b/a\|_{A_w} \, .
$$

\section{Proof of Theorem \ref{thm decay}}\label{sec decay-smoothness}
For $s \in \R$ let $J^s$ be the fractional derivative operator given by
$$J^s f (\theta) = \sum_{n\in \Z} (1+|n|^2)^{s/2} \hat{f}(n) e^{in\theta} \, ,$$
for $f\in L^2(\T)$. Note, that
$$
\|f\|_{H^s(\T)} = \| J^s f\|_{L^2(\T)} \, .
$$
We will need the fractional Leibniz rule of Kato-Ponce \cite{Kato1988CommutatorEA}. 
For $s>0$ and two functions $u,v \in H^s(\T)$, it claims
\begin{equation}\label{leibniz rule}
    \| J^s (uv) - uJ^sv \|_{L^2(\T)} \lesssim \|J^s u\|_2 \|v\|_\infty + \| u'\|_{L^\infty(\T)} \| J^{s-1}v\|_{L^2(\T)} \, .
\end{equation}

One can see from the recursions \eqref{recursion for a} and \eqref{recursion for b}, that
$$
F_n = \frac{(z^{-n}b_n) (\infty)}{a_n(\infty)} \, .
$$
This is sometimes called the layer-stripping formula \cite[page 17]{QSP_NLFA}. We can estimate
$$
|F_n| \leq \frac{|z^{-n}b_n(\infty)|}{a_n(\infty)} = \frac{1}{a_n(\infty)} \left| \frac{1}{2\pi} \int_{-\pi}^{\pi} e^{-in\theta} b_n(e^{i\theta}) d\theta \right|
$$
\begin{equation}\label{estimate through b_n sobolev}
    \leq \frac{1}{a(\infty)} |\hat{b_n}(n)| \leq \frac{1}{a(\infty)} |n|^{-s} \|b_n\|_{H^s(\T)} \, .    
\end{equation}
We need to estimate the Sobolev norm of $b_n$. It is straightforward that $J^s$ commutes with projection operators $P_I$. So from \eqref{rh 1} and \eqref{rh 2}, we get
$$
J^s a_n^* = a_n^*(0) - P_+\left( J^s(\frac{b^*}{a}b_n ) \right) 
$$
$$= a_n^*(0) - P_+ \left( J^s (\frac{b^*}{a}b_n) - \frac{b^*}{a} J^s b_n \right) - P_{+} \left( \frac{b^*}{a} J^s b_n \right) \, ,
$$
and
$$
J^s b_n = P_{\leq n } \left( J^s(\frac{b}{a^*} a_n^*) \right) = P_{\leq n}\left( J^s(\frac{b}{a^*}a_n^*) - \frac{b}{a^*} J^s a_n^* \right) + P_{\leq n} \left( \frac{b}{a^*} J^s a_n^* \right) \, .
$$
In the matrix form the two equations become
$$
(\Id_2+M)\begin{pmatrix}
    J^s a^*_n \\
    J^s b_n
\end{pmatrix}
 = \begin{pmatrix}
    a_n^*(0) - P_+ \left( J^s (\frac{b^*}{a}b_n) - \frac{b^*}{a} J^s b_n \right) \\
    P_{\leq n } \left( J^s (\frac{b}{a^*}a_n^*) -\frac{b}{a^*}J^s a_n^* \right)
\end{pmatrix} \, .
$$
Using \eqref{inverse op est}, we estimate

$$
\left\| \begin{pmatrix}
    J^s a_n^* \\
    J^s b_n
\end{pmatrix} \right\|_{L^2(\T)\times L^2(\T)} \leq \left\| \begin{pmatrix}
    a_n^*(0) - P_+ \left( J^s (\frac{b^*}{a}b_n) - \frac{b^*}{a} J^s b_n \right) \\
    P_{\leq n } \left( J^s (\frac{b}{a^*}a_n^*) -\frac{b}{a^*}J^s a_n^* \right) \, .
\end{pmatrix} \right\|_{L^2(\T)\times L^2(\T)}
$$
$$
\lesssim |a_n^*(0)| + \left\|J^s(\frac{b^*}{a}b_n)-\frac{b^*}{a} J^sb_n\right\|_{L^2(\T)} + \left\| J^s (\frac{b}{a^*}a^*_n) - \frac{b}{a^*} J^s a_n^* \right\|_{L^2(\T)} \, ,
$$
applying \eqref{leibniz rule} and recalling that $|a_n|,|b_n|\leq 1$ on $\T$, we can continue,
\begin{equation}\label{fr leibniz applied}
    \leq |a_n^*(0)| + 2\left\| J^s (\frac{b^*}{a}) \right\|_{L^2(\T)} + \left\|\left(\frac{b^*}{a}\right)'\right\|_{L^\infty(\T)} \left( \| J^{s-1} a_n^*\|_{L^2(\T)} + \| J^{s-1}b_n\|_{L^2(\T)} \right) \, .
\end{equation}
Let us iterate the above estimate $\lceil s\rceil$ times. We get
$$
\| J^s a_n^* \|_{L^2(\T)} + \|J^s b_n\|_{L^2(\T)} \lesssim_s (|a_n^*(0)+2\|J^s(b^*/a)\|_{L^2(\T)} ) \sum_{j=0}^{\lceil s\rceil -1} \| (b^*/a)'\|_{L^\infty (\T)}^j 
$$
$$
+ \| (b^*/a)'\|_{L^\infty(\T)}^{\lceil s\rceil} ( \| J^{s-\lceil s\rceil }a_n^*\|_{L^2(\T)} + \|J^{s-\lceil s\rceil } b_n\|_{L^2(\T)} ) \, .
$$
Note, that
$$\|J^{s-\lceil s\rceil}a_n^*\|_{L^2(\T)}\leq \|a_n^*\|_{L^2(\T)}\leq \|a_n^*\|_{L^\infty(\T)} \leq 1 \, ,$$
and similarly $\|J^{s-\lceil s\rceil }b_n\|_{L^2(\T)}\leq 1$. Therefore, we conclude
$$
\| J^s a_n^* \|_{L^2(\T)} + \|J^s b_n\|_{L^2(\T)} \lesssim_s (1+\|J^s(b^*/a)\|_{L^2(\T)} ) \max (1, \|b^*/a\|_{L^{\infty}(\T)}^{\lceil s\rceil}) \, .
$$
This concludes the estimate for $s>1$. For $s=1$, in \eqref{estimate through b_n sobolev} we can substitute the Sobolev norm of $b_n$ by its homogeneous variant. That is, we can write
$$
|F_n| \leq \frac{1}{|n| a(\infty)} \| b_n'\|_{L^2(\T)}\, .
$$
Then,
$$
(a_n^*) ' = - P_+ \left( (\frac{b^*}{a} b_n )' \right) = - P_+ \left( (\frac{b^*}{a})' b_n\right) - P_+ \left( \frac{b^*}{a^*} b_n' \right),
$$
$$
b_n' = P_{\leq n} \left( (\frac{b}{a^*}a_n^*)' \right) = P_{\leq n} \left( (\frac{b}{a^*})' a_n^* \right) + P_{\leq n} \left( \frac{b}{a^*} (a_n^*)' \right).
$$
Hence,
$$
(\Id_2+M)\begin{pmatrix}
    (a_n^*)' \\ b_n' 
    \end{pmatrix} = \begin{pmatrix}
        - P_{+} ( (\frac{b^*}{a})' b_n) \\
        P_{\leq n} ( (\frac{b}{a^*})' a_n^* )
    \end{pmatrix}.
$$
Again, using \eqref{inverse op est} and that $a_n$ and $b_n$ are bounded by $1$, estimate
\begin{equation}\label{eq:b2}
    \left \| \begin{pmatrix}
    (a_n^*)' \\ b_n' 
    \end{pmatrix} \right\|^2 \leq 2 \left \| \left( \frac{b}{a^*} \right)' \right \|_{L^2({\T})}^2 \, .
\end{equation}
This proves \eqref{firstorder decay}.

\bibliographystyle{amsalpha}

\bibliography{references}

\end{document}